\documentclass[a4paper, 11pt]{amsart}
\usepackage{amssymb, amsthm, amsfonts, amsmath, bm}
\usepackage{graphicx}%
\usepackage{multirow}%
\usepackage{mathrsfs}%
\usepackage[title]{appendix}%
\usepackage{xcolor}%
\usepackage{textcomp}%
\usepackage{manyfoot}%
\usepackage{booktabs}%
\usepackage{algorithm}%
\usepackage{algorithmicx}%
\usepackage{algpseudocode}%
\usepackage{listings}%
\usepackage{dsfont}
\usepackage[all]{xy}
\usepackage[shortlabels]{enumitem}
\usepackage{hyperref, color}
\usepackage{commath}
\usepackage{esdiff}
\usepackage{graphicx,subfigure}
\usepackage[center]{caption}
\usepackage{tikz}
\usetikzlibrary{matrix,arrows,decorations.pathmorphing}   
\usepackage{diagbox}
\usepackage{multirow}
\usepackage{tikz-cd}
\usepackage{esdiff}
\usepackage[pagewise]{lineno}
\usepackage[utf8]{inputenc}

\theoremstyle{thmstyleone}
\newtheorem{theorem}{Theorem}
\newtheorem{proposition}[theorem]{Proposition}
\newtheorem{lemma}[theorem]{Lemma}
\newtheorem{corollary}[theorem]{Corollary}

\theoremstyle{thmstyletwo}

\newtheorem{remark}{Remark}

\theoremstyle{thmstylethree}
\newtheorem{definition}{Definition}

\newcommand{\Rbb}{\mathds{R}}

\newcommand{\Sbb}{\mathds{S}}

\newcommand{\Zbb}{\mathds{Z}}

\DeclareMathOperator{\dist}{\mathrm{dist}}

\DeclareMathOperator{\Hom}{\mathrm{Hom}}
\DeclareMathOperator{\im}{\mathrm{Im}}

\newtheorem{introthm}{Theorem}[section]

\newtheorem{introprp}[introthm]{Proposition}

\newtheorem*{acknowledgements}{Acknowledgements}

\numberwithin{equation}{section}
\numberwithin{figure}{section}

\begin{document}

\title[Zoll manifolds with boundary]{Zoll manifolds with boundary}

\author[E. Longa, P. Piccione, and R. Santos]{Eduardo Longa, Paolo Piccione, and Roney Santos}

\address[E. Longa]{
	Institute of Mathematics, Statistics and Computer Science
	\newline\indent
	University of S\~ao Paulo
	\newline\indent 
	05508-090, S\~ao Paulo-SP, Brazil}
\email{\href{mailto:edulonga@ime.usp.br}{edulonga@ime.usp.br}}

\address[P. Piccione]{
	Department of Mathematics, 
	School of Sciences\newline\indent Great Bay University
	\newline\indent 
	523000, Dongguan-GD, People’s Republic of China
	\newline\indent 
	and
	\newline\indent
	(Permanent address) Institute of Mathematics, Statistics and Computer Science
	\newline\indent
	University of S\~ao Paulo
	\newline\indent 
	05508-090, S\~ao Paulo-SP, Brazil}
\email{\href{mailto:piccione@ime.usp.br}{paolo.piccione@usp.br}}

\address[R. Santos]{Department of Mathematics
\newline\indent King's College London
\newline\indent Strand, London WC2R 2LS, United Kingdom.
\newline\indent 
and
\newline\indent
Institute of Mathematics, Statistics and Computer Science
\newline\indent
University of S\~ao Paulo
\newline\indent 
05508-090, S\~ao Paulo-SP, Brazil}
\email{\href{mailto:roney.santos@kcl.ac.uk}{roney.santos@kcl.ac.uk}}

\begin{abstract}
We introduce and study Zoll manifolds with boundary: compact Riemannian manifolds with smooth boundary such that every geodesic issuing orthogonally from the boundary returns orthogonally and is nowhere tangent to it. We first show that all such free boundary geodesics are embedded and have a common length, and that the boundary has at most two connected components. If there are two components, we prove that the manifold is a product of an interval with a closed manifold. When the boundary is connected, we show that the manifold is a tubular neighborhood of a closed embedded submanifold, the ``soul'', and that the complement of the soul is diffeomorphic to a half-open cylinder over the boundary. We further prove that all free boundary geodesics are maximally degenerate critical points of the energy functional and have the same Morse index, which equals the multiplicity of the unique focal point occurring at the midpoint of each geodesic. The projection from the boundary to the soul is then either a nontrivial two-fold covering or a smooth sphere bundle, according to the value of this index. As applications, we obtain a complete classification of Zoll surfaces with boundary and of three-dimensional Zoll manifolds with boundary.
\end{abstract}



\maketitle

\section{Introduction}
Closed Riemannian manifolds all of whose geodesics are closed, often called $C$-manifolds, have been thoroughly investigated in the literature, see the celebrated book \cite{Besse}. Within this class, a special role is played by manifolds all of whose closed geodesics have the same minimal period, called Zoll manifolds.

The first nontrivial example of such manifolds was found in 1903 by Otto Zoll \cite{zoll1903uber}, a student of David Hilbert, who was looking for examples of surfaces, other than the round sphere, with the property that all geodesics are closed with the same length, and discovered spheres of revolution with the desired property.

Subsequently, Paul Funk \cite{funk} gave a sufficient condition for the existence of a deformation of the round metric on the $2$-sphere through Zoll metrics: if $e^{\rho_t} g_0$ is such a deformation, then $\dot{\rho}$ must be an odd function on the sphere. This condition was later shown to be sufficient in general by Victor Guillemin \cite{guillemin}.

An old question in the field is whether the assumption that all geodesics have the same length can be removed. In other words, one asks whether every $C$-manifold must necessarily be Zoll. However, as observed by John Olsen \cite{olsen}, lens spaces equipped with their canonical metrics provide examples of $C$-manifolds that are not Zoll. A conjecture often attributed to Marcel Berger asserts that every simply connected $C$-manifold is Zoll, a statement that has indeed been confirmed for spheres of all dimensions other than three \cite{gromollgrove, radeschiwilkingsn}.

As a natural extension of the notion of $C$-manifold, in this paper we study compact $n$-dimensional Riemannian manifolds $(M,g)$ with boundary $\partial M$, $n\geq2$, having the property that from every point of $\partial M$ there issues a free boundary geodesic that returns orthogonally to the boundary and is never tangent to it. Under this assumption, a first result we obtain is that the length of all free boundary geodesics is a fixed positive number. Thus, the free boundary analogues of closed $C$-manifolds and Zoll manifolds are equivalent, which means that the Berger conjecture is, in fact, true in this setting. For this reason, we call these manifolds \emph{Zoll manifolds with boundary}.

Our analysis shows that Zoll manifolds with boundary have a more rigid geometric structure than $C$-manifolds, as illustrated by the following paradigmatic example. Consider a closed manifold $\Sigma$ embedded in a Riemannian manifold $N$, with codimension $k+1$, $k\ge0$. For $d>0$, denote by $\nu_d(\Sigma)$ the radius $d$ normal disk bundle of $\Sigma$ in $N$. For sufficiently small $L>0$, the normal exponential map along $\Sigma$ carries $\nu_L(\Sigma)$ to a tubular neighborhood $M$ of $\Sigma$, having smooth boundary. In this situation, every inward-pointing geodesic segment of length $2L$ issuing orthogonally from a point of $\partial M$ remains inside $M$ and hits the boundary orthogonally at the final endpoint. Note that if $k>0$, then $\partial M$ is connected, while when $k=0$, $\partial M$ has at most two connected components.

Let us now consider a general Zoll manifold with boundary $(M,g)$.
First, we show that the boundary of $M$ can have at most two connected components, the most interesting case occurring when $\partial M$ is connected. We show that, as in the example above, $M$ can be described as a tubular neighborhood of an embedded closed submanifold of codimension at least $1$, called the \emph{soul} of $M$. 

\begin{introthm}\label{thm:maingeomstruct}
	Let $(M,g)$ be a Zoll manifold with boundary. Then: 
	\begin{itemize}
		\item all of its free boundary geodesics are embedded and have the same length $2L$;
		\item $\partial M$ has at most two connected components.
	\end{itemize} 
	Moreover:
	\begin{itemize}
		\item[$\mathrm{(i)}$] if $\partial M$ is connected, then there exists a smooth closed embedded submanifold $\Sigma_L$  of $M$ (the \emph{soul} of $M$) of codimension at least $1$, such that $M\setminus\Sigma_L$ is foliated by hypersurfaces $\Sigma_t$ equidistant from $\Sigma_L$, $t\in\left[0,L\right[$, each of which is diffeomorphic to $\partial M$. Here, $\Sigma_0$ coincides with $\partial M$, $\mathrm{dist}(\Sigma_t,\Sigma_L)=L-t$, and the metric $g$ on $M\setminus\Sigma_L\cong\left[0,L\right[\times\partial M$ has the form $\mathrm dt^2\oplus g_t$, where $\left[0,L\right[\ni t\mapsto g_t$ is a smooth one-parameter family of metrics on $\partial M$.
		\item[$\mathrm{(ii)}$] If $\partial M$ has two connected components, then the two components are both diffeomorphic to a closed manifold $\Sigma$ and $M$ is diffeomorphic to the product $[0,2L]\times\Sigma$, with metric $g\simeq \mathrm dt^2\oplus g_t$, where $\left[0,2L\right]\ni t\mapsto g_t$ is a smooth one-parameter family of metrics on $\Sigma$.
	\end{itemize}
\end{introthm}
As we have seen in the tubular neighborhood example above, any closed Riemannian manifold can be realized as the soul of a Zoll manifold with connected boundary in sufficiently high codimension. A certain cohomological condition is required if one requires codimension $1$ (Proposition~\ref{thm:soulcodim1connboundary}).

A description of the submanifold $\Sigma_L$ can be given in terms of focal points of the boundary, and this requires an investigation of the Morse index of the free boundary geodesics.

In the closed setting, Bott--Samelson's theorem \cite{Besse, Bott, Samelson} asserts that any two geodesics in a closed Zoll manifold have the same Morse index, and this fixed number directly influences the topology of the manifold. In the free boundary setting, we show that an analogous result holds (with a proof that also covers the case of closed Zoll manifolds, see Remark~\ref{rem:maxdegen}). 

\begin{introthm}\label{thm:sameMorseindex1}
	Any two free boundary geodesics of an $n$-dimensional Zoll manifold $(M,g)$ with boundary have the same Morse index. More specifically, given a unit-speed free boundary geodesic $\gamma\colon[0,2L]\to M$ in $M$, its Morse index is equal to the multiplicity of $\gamma(L)$ as a focal point of the boundary.
	When $\partial M$ has two boundary components, this number is zero.
\end{introthm}

The soul $\Sigma_L$ of a Zoll manifold with connected boundary, described in Theorem~\ref{thm:maingeomstruct}, is the set of midpoints of all free boundary geodesics. When $k>0$, $\Sigma_L$ can also be described as the focal set of $\partial M$, and $k$ is precisely the multiplicity of each focal point of $\partial M$. It is a smooth closed connected embedded submanifold of codimension $k+1$, where $k$ is the index of the free boundary geodesics. 

Unlike the closed case, where the Morse index is highly restricted, here the Morse index may take any value in $\{0,\ldots,n-1\}$, where $n$ is the dimension of the manifold:
\begin{introprp}\label{thm:anynk}
	For every $n\ge2$ and any $k\in\{0,\ldots,n-1\}$ there exists an
	$n$-dimensional Zoll manifold with connected boundary whose free boundary
	geodesics have Morse index equal to $k$.
\end{introprp}

In the last part of the paper, we describe the topology of Zoll manifolds with connected boundary.  

Inspired by Bott--Samelson's theorem, we show the following about the homology of a Zoll manifold with boundary.

\begin{introthm}\label{thm:hombordo}
	Let $(M,g)$ be a Zoll manifold with connected boundary and index $k \geq 0$. Then the relative homology group $H_1(M,\partial M;\Zbb)$ is either trivial or isomorphic to $\Zbb_2$. Moreover, $H_1(M,\partial M;\Zbb)\cong \Zbb_2$ if and only if $k = 0$.
\end{introthm}

A topological classification of Zoll surfaces with boundary follows immediately from Theorem \ref{thm:hombordo}: if the boundary has two connected components, the surface is diffeomorphic to a cylinder; if the boundary is connected, then the surface is diffeomorphic either to a M\"obius band, when $k = 0$, whose soul is a circle, or to a disk, when $k = 1$, whose soul is a point.

Now consider the map 
\begin{equation*}
	\mathcal{F}_L\colon\partial M\rightarrow\Sigma_L,
\end{equation*}
which carries every point $p\in\partial M$ to the unique intersection of the free boundary geodesic through $p$ with the soul. As the following result shows, this map endows $\partial M$ with the structure of a sphere bundle.

\begin{introthm} \label{thm:fiberbundle}
	Let $(M,g)$ be an $n$-dimensional Zoll manifold with connected boundary whose free boundary geodesics all have Morse index equal to $k$. One of the following holds:
	\begin{itemize}
		\item[$\mathrm{(i)}$] either $k=0$ and $\mathcal{F}_L \colon \partial M \to \Sigma_L$ is a nontrivial $2$-fold covering map;
		\item[$\mathrm{(ii)}$] or $k>0$ and $\mathcal{F}_L \colon \partial M \to \Sigma_L$ is a smooth fiber bundle whose fibers are $k$-spheres.
	\end{itemize}
	In particular, $\Sigma_L$ cannot be a closed simply connected manifold when $k = 0$.
\end{introthm}

Using Theorem~\ref{thm:fiberbundle} and standard classifications of disk bundles over circles and interval bundles over surfaces, we obtain a topological classification of $3$-dimensional Zoll manifolds with boundary, see Proposition~\ref{thm:class3dim}.

In order to generalize the idea of a \emph{curve of constant width} in $\Rbb^2$, Stewart Robertson \cite{Robertson1, Robertson2} introduced the concept of \emph{transnormal manifold}: a convex, connected, closed and embedded submanifold $\Sigma$ of $\Rbb^n$ such that, if $N_p\Sigma$ denotes the normal space of $\Sigma$ at $p \in \Sigma$, then, for each pair of points $p,q \in \Sigma$ we have that $q \in N_p\Sigma$ implies $N_p\Sigma = N_q\Sigma$. Later, John Bolton \cite{Bolton} generalized this concept to arbitrary ambient manifolds.

Notice that the boundary of a Zoll manifold with boundary is a transnormal hypersurface provided it is convex.  Several results about transnormal manifolds may be found in the literature, for instance, in the works above as well as in those of Seiki Nishikawa \cite{Nishikawa1, Nishikawa2}. We remark that we never assume convexity of the boundary, and that the results in our paper are original, to the best of our knowledge.

Proofs of the above results are given in the remainder of the paper. The first statements of Theorem~\ref{thm:maingeomstruct} are proven in Propositions~\ref{thm:Lp_loc_constant} and~\ref{thm:atmost2conncomp}. Part (ii) is proven in Theorem~\ref{Zolltwoboundaries}, and part (i) follows from Theorem~\ref{Zolloneboundary} and Proposition~\ref{SigmaL}. Theorem~\ref{thm:sameMorseindex1} follows from Corollary~\ref{thm:minimizer} and Proposition~\ref{thm:sameMorseindex}.

Proposition~\ref{thm:anynk} is proved in Section~\ref{sec:construction} using a mapping torus construction.

Theorem~\ref{thm:hombordo} is proved in Section~\ref{sec:topZollconnbound}, where we also obtain a topological classification of $2$ and $3$-dimensional Zoll manifolds with boundary. A particularly interesting case is that of Zoll $3$-manifolds with connected boundary when $k=0$, where the soul can be diffeomorphic to any non-simply connected closed surface, possibly nonorientable (Proposition~\ref{thm:class3dim}).

Finally, Theorem~\ref{thm:fiberbundle} is proved in Section \ref{projection}.

\section{Zoll manifolds with boundary and length of free boundary geodesics}\label{sec:length}
Let $(M,g)$ be an $n$-dimensional connected compact Riemannian manifold with smooth boundary $\partial M$. 
We introduce the following terminology: for $q\in\partial M$, an \emph{extension} of $(M,g)$ at $q$ is an $n$-dimensional smooth Riemannian manifold $(\widetilde M,\widetilde g)$ such that $M$ is embedded in $\widetilde M$, $q\in\widetilde M\setminus\partial\widetilde M$, and $\widetilde g\big\vert_M=g$.
Clearly, extensions of $(M,g)$ exist at every point $q\in\partial M$.

Let $\nu$ denote the (smooth) unit inward-pointing normal vector field to $\partial M$. For all $p\in\partial M$, let $\gamma_p$ denote the unit-speed orthogonal geodesic $t\mapsto\exp_p(t\cdot\nu_p)$, defined on some interval $[0,R_p]$ (or on the half-line $\left[0,+\infty\right[$).
\begin{definition}\label{thm:defzoll}
	A compact Riemannian manifold $(M,g)$ with boundary is said to be \emph{Zoll with boundary} if for all $p\in\partial M$, $R_p$ is the first positive instant when $\gamma_p$ meets $\partial M$, and $\gamma_p'(R_p)\in T_{\gamma_p(R_p)}(\partial M)^\perp$.
\end{definition}

In other words, $(M,g)$ is Zoll with boundary if every geodesic that starts orthogonally from the boundary remains nowhere tangent to it and meets the boundary again orthogonally.

The assumption that $\gamma_p$ is nowhere tangent to $\partial M$ ensures that the first
return time function $\partial M \ni p \longmapsto R_p \in \Rbb$ is smooth, as shown in the proof of Proposition \ref{thm:Lp_loc_constant} below. If tangency points are allowed, pathological situations such as the one depicted in Figure \ref{fig:tangency} may occur, where $R$ is not even continuous. Note that in this example, not all geodesics have the same length. In fact, we do not know whether examples exist in which all orthogonal geodesics have the same length while some of them become tangent to the boundary at an intermediate point.

\begin{figure}[h]
	\centering
	\includegraphics[width=8cm]{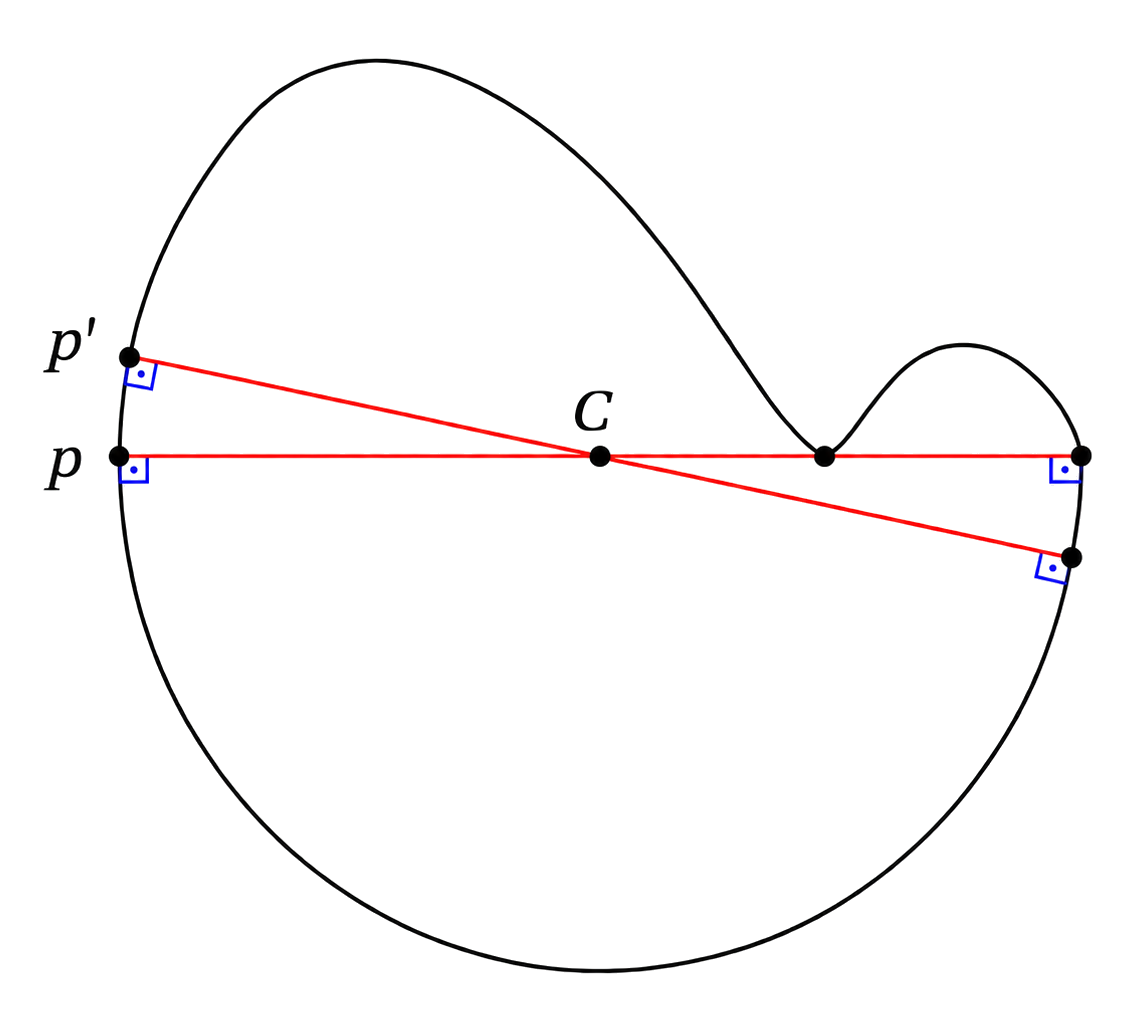}
	\caption{An example of a disk where an orthogonal geodesic becomes tangent to the boundary at an intermediate point. In this case, the first return time function $R$ fails to be continuous at $p$.}
	\label{fig:tangency}
\end{figure}
\begin{proposition}\label{thm:Lp_loc_constant}
	Let $(M,g)$ be a Zoll manifold with boundary. Then, the map $\partial M\ni p\mapsto R_p\in\mathds R$ is locally constant.
\end{proposition}
\begin{proof}
	First, let us observe that the map $\partial M\ni q\mapsto R_q\in\left]0,+\infty\right[$ is smooth. Namely, let $p\in \partial M$ be fixed, and let $(\widetilde M,\widetilde g)$ be an \emph{extension} of $(M,g)$ at $\gamma_p(R_p)$. 
	
	Consider the smooth map $\mathcal{F}$ given by $(q,t)\mapsto\gamma_q(t)\in\widetilde M$,
	defined for $(q,t)$ in a neighborhood of $(p,R_p)$ in $\partial M\times\mathds R$. This map $\mathcal{F}$ is transverse to $\partial M$ at $(p,R_p)$, because: 
	\[\frac{\partial \mathcal{F}}{\partial t}(p,R_p)=\gamma_p'(R_p)\in T_{\gamma_p(R_p)}(\partial M)^\perp\setminus\{0\}.\]
	Thus, given a sufficiently small neighborhood $U$ of $\gamma_p(R_p)$ in $\widetilde M$, the preimage $\mathcal{F}^{-1}(U\cap\partial M)$ is a smooth submanifold of $\partial M\times\mathds R$. This preimage is the graph of the map $\partial M\ni q\mapsto R_q\in\mathds R$, with $q$ near $p$, which proves that $q\mapsto R_q$ is smooth.
	
	Now, let us compute the differential of the map $q\mapsto R_q$. Given $p\in\partial M$ and $v\in T_p(\partial M)$, let $\left]-\varepsilon,\varepsilon\right[\ni s\mapsto p_s\in \partial M$ be a smooth map with $p_0=p$ and $p'_0=v$; for all $s$, let $\gamma_s\colon[0,1]\to M$ denote the affine reparametrization of $\gamma_{p_s}\big([0,R_{p_s}]\big)$ on the interval $[0,1]$. Then, $\gamma_s$ is a smooth variation of $\gamma_0$. Note that $g(\gamma_s',\gamma_s')=R_{p_s}^2$. Differentiating, we obtain:
	\[\frac12\,\frac{\mathrm d}{\mathrm ds}\Big\vert_{s=0}\;g(\gamma_s',\gamma_s')=g(J',\gamma_0'),\]
	where $J$ is the Jacobi field along $\gamma_0$ given by $\frac{\mathrm d}{\mathrm ds}\big\vert_{s=0}\gamma_s$, and $J'$ is its covariant derivative. By differentiating, it is easy to see that $g(J',\gamma_0')$ is constant on $[0,1]$; let us show that this constant is zero. Observe that since $\gamma_s$ is a free boundary geodesic in $M$, then $J(0)$ and $J(1)$ are tangent to $\partial M$. So, integration by parts gives:
	\[\int_0^1g(J',\gamma_0')\,\mathrm dt=-\int_0^1g\big(J,\tfrac{\mathrm D}{\mathrm dt}\gamma_0'\big)\,\mathrm dt+g(J,\gamma_0')\big\vert_0^1=0,\]
	because $\tfrac{\mathrm D}{\mathrm dt}\gamma_0'\equiv0$, and because $\gamma_0'$ is orthogonal to $\partial M$ at the endpoints.
	This proves that the map $q\mapsto R_q$ has vanishing differential, i.e., it is locally constant on $\partial M$.
\end{proof}
In fact, one can show that in a Zoll manifold with boundary, the map $p\mapsto R_p$ is constant. In other words, the Zoll condition implies that all free boundary geodesics have the same length. Thus, the free boundary analogues of closed $C$-manifolds and Zoll manifolds coincide.

\begin{proposition}\label{thm:atmost2conncomp}
	Let $(M,g)$ be a Zoll manifold with boundary. Then $\partial M$ has at most two connected components, and the first return function $\partial M\ni p\mapsto R_p\in\mathds R$ is constant; in particular, all free boundary geodesics have the same length. If $M$ has two boundary components, then they are diffeomorphic.
\end{proposition}
\begin{proof}
	Denote by $C$ a connected component of $\partial M$.
	The map $C\ni p\mapsto \gamma_p(R_p)\in\partial M$ is continuous, and therefore it takes values in some fixed connected component $C'$ of $\partial M$. There cannot exist a connected component $C''$ of $\partial M$ which is distinct from $C$ and $C'$. Namely, any curve of minimal length between $C$ and $C''$ (which exists by the compactness and connectedness assumptions on $M$) is a free boundary geodesic with endpoints in $C$ and $C''$, while we have just observed that any free boundary geodesic with one endpoint in $C$ must have its other endpoint in $C'$. Thus, $\partial M$ has at most two connected components. 
	
	Now, assuming that $\partial M$ has exactly two connected components, by a similar argument we obtain that every free boundary geodesic in $M$ must have its endpoints lying in different connected components of $\partial M$. This observation leads to the conclusion that the map $p\mapsto R_p$, which is in fact the length of the free boundary geodesic starting at $p$, must take the same constant value on each of the two connected components of $\partial M$. Moreover, the map $\partial M\ni p\mapsto\gamma_p(R_p)\in\partial M$ is an involution, which carries diffeomorphically one of the two connected components of $\partial M$ to the other.
\end{proof}

From the proof of Proposition~\ref{thm:atmost2conncomp} we obtain:

\begin{corollary}\label{thm:minimizer}
	In a Zoll manifold $(M,g)$ with two boundary components, every free boundary geodesic minimizes the distance between these two connected components.\qed
\end{corollary}
It follows in particular that there are no $\partial M$-focal points in $M$ when  $(M,g)$ is a Zoll manifold with disconnected boundary.

\section{Degeneracy and Morse index}\label{sec:degMorseInd}
The main references for this section are the papers \cite{Kal88, PicTau99}.
Free boundary geodesics in a Riemannian manifold with boundary $(M,g)$ are critical points of the energy functional $E(\gamma)=\frac12\int_0^1g(\gamma',\gamma')\,\mathrm dt$ defined in the space $\mathcal H(M,\partial M)$ of curves $\gamma\colon[0,1]\to M$ of Sobolev regularity $H^1$, and satisfying the boundary condition $\gamma(0),\gamma(1)\in\partial M$. It is well known that critical points of $E$ in $\mathcal H(M,\partial M)$ have finite Morse index. A Morse index theorem for this situation can be easily obtained from the Morse index theorem in the case of both endpoints varying in submanifolds, see for instance \cite{PicTau99} for the statement and a concise proof of this result.   

Recall that the Morse index of a free boundary geodesic $\gamma\colon[0,1]\to M$ is given by the index of the symmetric bilinear form
\begin{multline*}
	I_\gamma(V,W)=\int_0^1g(V',W')+g(\mathcal R_\gamma V,W)\,\mathrm dt \\ 
	+\mathcal S_{\gamma'(1)}\big(V(1),W(1)\big)-
	\mathcal S_{\gamma'(0)}\big(V(0),W(0)\big)
\end{multline*}
defined in the space $\mathcal H_\gamma$ of vector fields (of Sobolev class $H^1$) along $\gamma$ that are tangent to $\partial M$ at the endpoints. Here, the prime ${}'$ denotes the covariant derivative of fields along $\gamma$, $\mathcal S$ denotes the second fundamental form of $\partial M$, and $\mathcal R_\gamma=\mathcal R(\gamma',\cdot)\gamma'$ is the curvature tensor of the Levi--Civita connection of $g$, chosen with the sign convention $\mathcal R(X,Y)=[\nabla_X,\nabla_Y]-\nabla_{[X,Y]}$.

The kernel of $I_\gamma$ is the space of \emph{free boundary Jacobi fields} along $\gamma$, i.e., the set of vector fields $J$ along $\gamma$ satisfying the Jacobi equation $J''=\mathcal R_\gamma J$, and the boundary conditions:
\begin{eqnarray}
	&& J(0)\in T_{\gamma(0)}(\partial M),\quad J'(0)+\mathcal S_{\gamma'(0)}\big(J(0)\big)\in T_{\gamma(0)}(\partial M)^\perp\label{eq:freeboundJacobi1}\\
	&& J(1)\in T_{\gamma(1)}(\partial M),\quad J'(1)+\mathcal S_{\gamma'(1)}\big(J(1)\big)\in T_{\gamma(1)}(\partial M)^\perp .\label{eq:freeboundJacobi2}
\end{eqnarray}
We claim that the dimension of the kernel of $I_\gamma$ is, in general, less than or equal to $n-1$, with $n=\dim M$. Namely, the space of Jacobi fields satisfying only the initial conditions \eqref{eq:freeboundJacobi1} has dimension $n$. Moreover, the Jacobi field $J(t)=t\gamma'(t)$ satisfies \eqref{eq:freeboundJacobi1}, but not \eqref{eq:freeboundJacobi2}. From these two observations, the claim follows.

Let us denote by $\mathbb J_\gamma$ the space of Jacobi fields along $\gamma$ that satisfy the initial condition \eqref{eq:freeboundJacobi1}. This is easily seen to be a vector space of dimension $n$. An instant $t\in\left]0,1\right]$ is said to be \emph{focal along $\gamma$} if the map $\mathbb J_\gamma\ni J\mapsto J(t)\in T_{\gamma(t)}M$ is not surjective. In this case, the codimension of the image of this map is defined to be the \emph{multiplicity} of the focal instant. It is well known that the focal instants along $\gamma$ are isolated, and the multiplicity of each focal instant is less than or equal to $n-1$.

Let us now consider the space $\mathbb J_\gamma^\parallel$ consisting of Jacobi fields $J$ along $\gamma$ satisfying \eqref{eq:freeboundJacobi1} and $J(1)\in T_{\gamma(1)}(\partial M)$, and the symmetric bilinear form $\mathcal A_\gamma\colon\mathbb J_\gamma^\parallel\times\mathbb J_\gamma^\parallel\to\mathds R $, defined by:
\[\mathcal A_\gamma(J_1,J_2)=\mathcal S_{\gamma'(1)}\big(J_1(1),J_2(1)\big)+g\big(J_1'(1),J_2(1)\big).\]
The symmetry of $\mathcal A_\gamma$ is easily verified using the fact that, for $J_1,J_2\in\mathbb J_\gamma$, 
$g\big(J_1'(t),J_2(t)\big)=g\big(J_1(t),J_2'(t)\big)$ for all $t\in[0,1]$. Note that $\mathcal A_\gamma(J_1,J_2)=0$ for all $J_2\in\mathbb J_\gamma^\parallel$ if $J_1$ satisfies \eqref{eq:freeboundJacobi2}.

A Morse index theorem is available in this setting, see \cite{Kal88}, or \cite{PicTau99} for a more recent proof. If $t=1$ is not focal along $\gamma$, then the index of the bilinear form $I_\gamma$ in the space $\mathcal H_\gamma$ is equal to the number of focal instants along $\gamma$, counted with multiplicity, plus the index of the bilinear form $\mathcal A_\gamma$ in the space $\mathbb J_\gamma^\parallel$. %

Let us now analyze the situation in the case when $(M,g)$ is Zoll with boundary (Definition~\ref{thm:defzoll}).
\begin{proposition}\label{thm:maxdeg}
	In a Zoll manifold with boundary, every free boundary geodesic $\gamma$ in $M$ is \emph{maximally degenerate}, i.e.,  the kernel of $I_\gamma$ has maximal dimension $n-1$. Moreover, the Morse index of every free boundary geodesic $\gamma$ is equal to the number of $\partial M$-focal points along $\gamma$ (counted with multiplicity).
\end{proposition}
\begin{proof}
	This follows readily from the observation that, given any free boundary geodesic $\gamma\colon[0,1]\to M$ in $M$, one obtains an $(n-1)$-parameter family of deformations of $\gamma$ by free boundary geodesics (having the same length as $\gamma$) by varying the initial endpoint on  $\partial M$. 
	
	By the same argument, one easily deduces that the space $\mathbb J_\gamma^\parallel$ has maximal dimension $n-1$, and it coincides with the kernel of $I_\gamma$. Therefore, $\mathcal A_\gamma\equiv0$, and the Morse index theorem tells us that the index of $I_\gamma$ is equal to the number of focal points along $\gamma$.
\end{proof} 
From Proposition~\ref{thm:maxdeg}, we easily obtain the following.
\begin{proposition}\label{thm:sameMorseindex}
	Let $(M,g)$ be a Zoll manifold with boundary. Then, every two free boundary geodesics of $M$ have the same Morse index.
	This number is zero if $\partial M$ has two connected components.
\end{proposition} 
\begin{proof}
	Any two free boundary geodesics in $(M,g)$ are homotopic by a continuous homotopy of free boundary geodesics, which by Proposition \ref{thm:maxdeg} are maximally degenerate. A jump of the Morse index function can only occur\footnote{%
		A formal proof of this fact requires some functional-analytical arguments, based on the fact that, in the appropriate functional analytical settings, the index form of a continuous  path of free boundary geodesics is represented by a continuous path of \emph{essentially positive} self-adjoint operators on a Hilbert space. The number of negative eigenvalues of these operators can only change if the dimension of their kernels is not constant. The interested reader can find details in \cite[Section~2]{Gia01}.} if the dimension of the kernel of the corresponding index form is non-constant, which does not occur. 
	
	If $\partial M$ has two connected components, then each free boundary geodesic in $M$ minimizes the distance between these two components (Corollary~\ref{thm:minimizer}). Therefore, the Morse index is equal to $0$.
\end{proof}

\begin{remark}\label{rem:maxdegen}
	Note that the maximal degeneracy argument can also be applied to prove that in a closed Zoll manifold all prime closed geodesics have the same Morse index. The degeneracy of a periodic geodesic $\gamma$ in a closed Zoll manifold of dimension $n\ge2$ is equal to $2n-2$, which is the maximal dimension of the space of periodic Jacobi fields along $\gamma$.
\end{remark}

The result of Proposition~\ref{thm:sameMorseindex} suggests the following:
\begin{definition}
	The \emph{index} of a Zoll manifold with boundary is the Morse index of any of its free boundary geodesics.
\end{definition}

\section{Geometry of Zoll manifolds with boundary}
Let us fix some notation to be used in the remainder of the paper. 
Let $(M,g)$ be a Zoll manifold with boundary whose free boundary geodesics all have length $2L$. 
Denote by $\Sigma$ one of its boundary components. 
For each point $p \in \Sigma$, let 
$\gamma_p : [0,2L] \to M$
be the unit-speed free boundary geodesic of $M$ starting at $p$. 
Define the map 
\[
\mathcal{F} : \Sigma \times [0,2L] \to M, \qquad \mathcal{F}(p,t) = \gamma_p(t).
\] 
For each $t \in [0,2L]$, we set 
$\Sigma_t = \mathcal{F}(\Sigma,t)$. Notice that $\mathcal{F}$ is a surjective map.
\begin{proposition}\label{surjection}
	Each point in a Zoll manifold with boundary lies on a free boundary geodesic.
\end{proposition}

\begin{proof}
	Let $\Sigma$ be a boundary component of a Zoll manifold $M$ with boundary. For a given $q \in M$, one can find $p \in \Sigma$ minimizing the distance from $q$ to $\Sigma$. Thus, the free boundary geodesic starting at $p$ passes through $q$.
\end{proof}

\subsection{Zoll manifolds with disconnected boundary}
Let us start with the study of the geometry of Zoll manifolds with disconnected boundary.

\begin{theorem}\label{Zolltwoboundaries}
	Let $(M,g)$ be a Zoll manifold with two boundary components whose free boundary geodesics all have length $2L$. Then, all of its free boundary geodesics are embedded. Moreover, $M \simeq [0, 2L] \times \Sigma$, where $\Sigma$ is a connected component of $\partial M$, and $g\simeq \mathrm dt^2\oplus g_t$, where $[0,2L]\ni t\mapsto g_t$ is a smooth one-parameter family of Riemannian metrics on $\Sigma$.
\end{theorem}

\begin{proof}
	Since the free boundary geodesics minimize length between $\Sigma = \Sigma_0$ and $\Sigma_{2L}$, they also minimize length between $\Sigma_0$ and $\Sigma_t$ for all $t \in\ ]0,2L[$. In particular, each $\Sigma_t$ is a hypersurface of $M$ because $\Sigma_0$ has no focal points in $M$, and consequently any free boundary geodesic is embedded and orthogonal to $\Sigma_t$.
	
	We are going to show that $\mathcal{F}$ is a bijection. By Proposition \ref{surjection}, all we need to verify is that $\mathcal{F}$ is injective. If $\mathcal{F}(x_1,t_1) = \mathcal{F}(x_2,t_2)$, then $\gamma_{x_1}(t_1) \in \Sigma_{t_2}$ and consequently $t_1 = t_2 = t_0$. Moreover, if $x_1 \neq x_2$, we must have $\gamma_{x_1}(t_0) = \gamma_{x_2}(t_0) = p$ and $\gamma'_{x_1}(t_0) \neq \gamma'_{x_2}(t_0)$, which is a contradiction because these vectors are orthogonal to the hypersurface $\Sigma_{t_0}$ at $p$. Therefore, $M$ is diffeomorphic to $\Sigma \times [0,2L]$.
	
	As for the metric $g$, we first observe that writing $(t,x) = (x^0, x^1, \ldots, x^n)$ in local coordinates, we have
	\[g_{0i} = g\left(\frac{\partial \mathcal{F}}{\partial t}, \frac{\partial \mathcal{F}}{\partial x^i}\right) =
	\begin{cases}
		1,\ &\text{if}\ i = 0\\
		0,\ &\text{if}\ i = 1, \ldots, n
	\end{cases}\]
	by the orthogonality condition. Furthermore, by the geodesic equations, $\gamma_x(t) = (t, x^1, \ldots, x^n)$ is a geodesic of $M$ if and only if $\Gamma_{00}^k = 0$ for all $k = 0, \ldots, n$. Here, $\Gamma_{ij}^k$ denotes the Christoffel symbols of $g$. But 
	\[\Gamma_{00}^k = \frac{1}{2}\sum_{\ell=0}^n \left(2\frac{\partial}{\partial t}g_{0\ell} - \frac{\partial}{\partial x^\ell}g_{00}\right)g^{k\ell} = 0,\]
	which proves our assertion.
\end{proof}
\subsection{Zoll manifolds with connected boundary}
We now consider Zoll manifolds with connected boundary. Set $\Sigma = \partial M$, so that $\mathcal{F} :  \partial M \times [0, 2L] \to M$ and $\Sigma_t = \mathcal{F}(\partial M, t)$. Before proceeding, we need the following auxiliary results.

\begin{lemma}\label{distance}
	Let $(M,g)$ be a Zoll manifold with connected boundary whose free boundary geodesics all have length $2L$. Then $\Sigma_t = \Sigma_{2L-t}$ for all $t \in [0,2L]$. Moreover, the following assertions hold.
	\begin{itemize}
		\item[$\mathrm{(i)}$] If $ t \leq L$, then $\dist(\Sigma_t, \partial M) = t$.
		\item[$\mathrm{(ii)}$] If $t,s \leq L$ and $\Sigma_t \cap \Sigma_s \neq \emptyset$, then $t = s$.
	\end{itemize}
\end{lemma}

\begin{proof}
	Let $p' = \gamma_p(2L)$. Notice that $2L = 2t + \ell_{p,t}$, where $\ell_{p,t}$ denotes the length of $\gamma_p$ between $\gamma_p(t)$ and $\gamma_{p'}(t)$ for $t \leq L$. Hence, $t + \ell_{p,t} = 2L - t$, which means that $\gamma_p(t) = \gamma_{p'}(2L-t)$ and $\gamma_{p'}(t) = \gamma_p(2L-t)$. Therefore, $\Sigma_t = \Sigma_{2L-t}$.
	
	Let $d = \dist(\Sigma_t, \partial M)$. There exists $p \in \partial M$ and a unit-speed geodesic $\sigma_p: [0,d] \to M$ realizing the distance between $\partial M$ and $\Sigma_t$. By minimality, $\sigma_p$ is orthogonal to $\partial M$. Hence, since $\Sigma_t = \Sigma_{2L-t}$ and $t \leq L$, we get $d = \mathrm{length}(\sigma_p) = t$.
	
	As a consequence, each point of $\Sigma_t$ realizes the distance between itself and $\partial M$. Thus, if there is $q \in \Sigma_{t_1} \cap \Sigma_{t_2}$, we get
	\[t_i = \dist(\Sigma_{t_i}, \partial M) = \dist(q, \partial M)\]
	for $i = 1,2$, which shows that $t_1 = t_2$.
\end{proof}

Since the set of midpoints of the free boundary geodesics of a Zoll manifold with connected boundary will play an important role in our study, we give the following definition.

\begin{definition}
	Given a Zoll manifold $(M,g)$ with boundary whose free boundary geodesics all have length $2L$, we call $\Sigma_L$ its {\it soul}.
\end{definition}

The first property of the soul of a Zoll manifold with boundary is that it contains all focal points of the free boundary geodesics.

\begin{lemma}\label{focalset}
	Let $(M,g)$ be a Zoll manifold with connected boundary. Then the focal set of $\partial M$ is contained in the soul of $M$.
\end{lemma}

\begin{proof}
	Since every free boundary geodesic of $M$ minimizes the distance between $\partial M$ and its soul by Lemma \ref{distance}, the map $\mathcal{F}_t: \partial M \to \Sigma_t$ given by $\mathcal{F}_t(p) = \mathcal{F}(p,t)$ is a diffeomorphism for all $t < L$.
\end{proof}

\begin{theorem}
	\label{Zolloneboundary}
	Let $(M,g)$ be a Zoll manifold with connected boundary whose free boundary geodesics all have length $2L$. Then all of its free boundary geodesics are embedded. Moreover, $M\setminus\Sigma_L$ is foliated by $\Sigma_t$, and it is diffeomorphic to the product $\left[0,L\right[\times\partial M$. On $M\setminus\Sigma_L\cong\left[0,L\right[\times\partial M$, the metric takes the form $g \simeq \mathrm dt^2\oplus g_t$, where $\left[0,L\right[\ni t\mapsto g_t$ is a smooth one-parameter family of metrics on $\partial M$.
\end{theorem}

\begin{proof}
	For any $s < L$, let $M_s$ denote the region of $M$ bounded by $\partial M$ and $\Sigma_s$, and $g_s = g|_{M_s}$. By Lemmas \ref{distance} and \ref{focalset}, we obtain that $(M_s,g_s)$ is a Zoll manifold with two boundary components that converges to $(M,g)$ as $s \to L$. It follows from Theorem \ref{Zolltwoboundaries} that each free boundary geodesic of $M_s$ is embedded, and consequently the same holds for the free boundary geodesics of $M$. Using the fact that $\Sigma_t$ is a smooth embedded hypersurface of $M$ for $t\in \left[0,L\right[$, one easily obtains that $\mathcal{F}_t\colon\partial M\to\Sigma_t$ is a diffeomorphism for all $t\in \left[0,L\right[$, that $M\setminus\Sigma_L\cong\left[0,L\right[\times\partial M$, and that the metric takes the form $g \simeq \mathrm dt^2\oplus g_t$, where $g_t$ is the metric on $\partial M$ given by the pullback of $g$ by $\mathcal{F}_t$.
\end{proof}

\section{The projection map}\label{projection}

This section is devoted to studying the properties of the natural projection $\mathcal{F}_L : \partial M\to\Sigma_L$ onto the soul. Denote by $\mathcal I\colon\partial M\to\partial M$ the smooth involution $p\mapsto\exp_p(2L\cdot\nu_p)$. Note that $\mathcal{F}_L$ is never injective, since $\mathcal{F}_L(p)=\mathcal{F}_L\big(\mathcal I(p)\big)$. For $p\in\partial M$, $\mathrm d\mathcal{F}_L(p)$ has a kernel whose dimension is equal to the multiplicity (possibly zero) of the focal instant $t=L$ along the free boundary geodesic $[0,2L]\ni t\mapsto\exp_p(t\cdot\nu_p)$. Using the constant rank theorem, one can show that $\Sigma_L$ is, in fact, a closed embedded smooth submanifold of $M$ whose dimension is related to the index of $M$.

\begin{proposition}\label{SigmaL}
	Let $(M,g)$ be an $n$-dimensional Zoll manifold with connected boundary and index $k$. Then its soul is a closed connected embedded submanifold contained in the interior of $M$ and has dimension equal to $n-1-k$. Moreover, one of the following holds:
	\begin{itemize}
		\item[$\mathrm{(i)}$] either $k=0$ and $\mathcal{F}_L \colon \partial M \to \Sigma_L$ is a $2$-fold covering map, which is nontrivial if and only if $\partial M$ is connected;
		\item[$\mathrm{(ii)}$] or $k>0$ and $\mathcal{F}_L \colon \partial M \to \Sigma_L$ is a smooth fiber bundle whose fibers are $k$-spheres.
	\end{itemize}
\end{proposition}

\begin{proof}
	Let us consider the case $k=0$ first. In this case, $\mathcal{F}_L$ is an immersion of $\partial M$ in the interior of $M$. Since $\mathcal{F}_L$ is not injective, the conclusion cannot be drawn from the immersion property. Instead, the conclusion will be obtained from the following claim: if $p_1,p_2\in \partial M$, and $\mathcal{F}_L(p_1)=\mathcal{F}_L(p_2)$, then there exist neighborhoods $\mathcal U_1,\mathcal U_2\subset\partial M$ of $p_1$ and $p_2$, respectively, such that $\mathcal{F}_L(\mathcal U_1)=\mathcal{F}_L(\mathcal U_2)$.
	In order to prove the claim, let us argue as follows. First, by the local form of immersions there exist neighborhoods $\mathcal U_1,\mathcal U_2\subset\partial M$ of $p_1$ and $p_2$, respectively, such that $\mathcal{F}_L(\mathcal U_i)$ is an embedded hypersurface of $M$ containing $q=\mathcal{F}_L(p_1)=\mathcal{F}_L(p_2)$. Locally, $M\setminus\mathcal{F}_L(\mathcal U_1)$ has two connected components, and every point of each of these connected components has distance strictly less than $L$ from $\partial M$. If no neighborhood of $p_1$ had its image contained in $\mathcal{F}_L(\mathcal U_2)$, then one could find points $q'\in\mathcal{F}_L(\mathcal U_1)$ arbitrarily close to $q$ whose distance from $\partial M$ would be less than $L$. This is impossible since all points in $\mathcal{F}_L(\partial M)=\Sigma_L$ have distance exactly equal to $L$ from $\partial M$.
	This proves the claim, and therefore $\Sigma_L$ is a compact connected hypersurface of the interior of $M$. Now, every free boundary geodesic in $M$ arrives orthogonally to $\Sigma_L$, by Gauss' Lemma. In this case, since $\Sigma_L$ has codimension $1$ in $M$, there cannot be more than two points in $\partial M$ with the same $\mathcal{F}_L$-image.  By the local form of immersions, $\mathcal{F}_L$ is a local diffeomorphism around every point of $\partial M$, and by compactness, it follows that $\mathcal{F}_L$ is a two-fold covering. This concludes the case $k=0$.
	
	Assume now that $k>0$. Let us first prove that $\mathcal{F}_L$ is a smooth map of constant rank equal to $n-1-k$. In order to prove this, first recall that every free boundary geodesic has Morse index equal to $k$. Since all focal points must lie in $\Sigma_L$ by Lemma \ref{focalset}, and each free boundary geodesic intersects $\Sigma_L$ exactly once, each free boundary geodesic has exactly one focal point of multiplicity $k$. This number is therefore the dimension of the kernel of $\mathrm d\mathcal{F}_L(p)$ at each point $p\in\partial M$, which implies that $\mathcal{F}_L$ has constant rank equal to $n-1-k$. 
	
	Using the local form of maps with constant rank, it follows immediately that $\Sigma_L$, which is the image of $\mathcal{F}_L$, is a smooth immersed submanifold, possibly with self-intersections, and that for all $p\in\Sigma_L$, $\mathcal{F}_L^{-1}(p)$ is a smooth embedded submanifold of $\partial M$. Now, we observe that for $t\ne L$, $\Sigma_t$ is precisely the set of points whose distance from $\Sigma_L$ is equal to $\vert t-L\vert$. In other words, $\Sigma_t$ is the boundary of the $\vert t-L\vert$-tubular neighborhood of $\Sigma_L$, and we know that for all $t\ne L$ this is a smooth embedded hypersurface. This easily implies that $\Sigma_L$ cannot have singularities due to self-intersections. Namely, the boundary of sufficiently small tubular neighborhoods of immersed submanifolds with singular points at self-intersections is not smooth. 
\end{proof}

\begin{remark}\label{rem:k=n-1}
	Note in particular that, when $k=n-1$, then $M$ is diffeomorphic to an $n$-ball of radius $L$, and the soul $\Sigma$ consists of a single point.
\end{remark}

As a direct consequence, in the case $k=n-2$, we get that $\partial M$ is foliated by spheres of codimension one. This, together with Thurston's theorem in \cite{thurston1976}, shows that the Euler characteristic of $\partial M$ vanishes.

\begin{corollary}\label{topology}
	Let $(M,g)$ be an $n$-dimensional Zoll manifold with connected boundary and index $k$. If $k=n-2$, then $\chi(\partial M) = 0$.
\end{corollary}

Also, from the proof of Proposition~\ref{SigmaL} we easily obtain:

\begin{corollary}\label{thm:fibresballs}
	Let $(M,g)$ be an $n$-dimensional Zoll manifold with connected boundary, soul $\Sigma_L$ and index $k$. The nearest-point map $\pi : M\to\Sigma_L$ is (well defined and) a smooth fiber bundle, whose fibers are $(k+1)$-dimensional closed balls.
\end{corollary}

We conclude this section with the following result on the relation between the index of $M$ and the intersection of its free boundary geodesics.

\begin{corollary} \label{intersect}
	If a Zoll manifold with connected boundary has index zero, then none of its free boundary geodesics intersect.
\end{corollary}

\begin{proof}
	Since $\mathcal{F}_t$ is a diffeomorphism for all $t < L$, an intersection point between two free boundary geodesics can only take place in the soul $\Sigma_L$. Since $\Sigma_L$ is a smooth hypersurface by Proposition \ref{SigmaL}, if two free boundary geodesics intersect at a point $p \in \Sigma_L$, then $\Sigma_L$ has two normal vectors at $p$, which is a contradiction.
\end{proof}

\section{Topology of Zoll manifolds with connected boundary}
\label{sec:topZollconnbound}
We begin by collecting here some results about Zoll manifolds with connected boundary. The first one deals with the homology of such manifolds.

\begin{theorem}\label{homology2}
	Let $(M,g)$ be a Zoll manifold with connected boundary and index $k \geq 0$. Then
	\begin{align*}
		H_1(M,\partial M;\Zbb) \cong 
		\begin{cases}
			0, & \text{if } k > 0 \\
			\Zbb_2, & \text{if } k = 0.
		\end{cases}
	\end{align*}
\end{theorem}

\begin{proof}
	Let $\Sigma_L$ denote the soul of $M$. By Corollary \ref{thm:fibresballs}, the nearest-point projection $\pi : M \to \Sigma_L$	is a disk bundle. In particular, since the fibers are contractible, $\pi$ is a homotopy	equivalence. Moreover, by Proposition \ref{SigmaL}, the restriction of $\pi$ to the	boundary coincides with the map	$\mathcal{F}_L: \partial M \to \Sigma_L$. Also let $i : \partial M \to M$ denote the inclusion map, so that $\mathcal{F}_L = \pi \circ i$. Consider the following part of the long exact sequence of the pair $(M,\partial M)$:
	
	\begin{equation} \label{longexact}
		\begin{tikzcd}
			\cdots \arrow[r] &
			H_1(\partial M;\Zbb) \arrow[r, "i_\ast"] &
			H_1(M;\Zbb) \arrow[r] &
			H_1(M, \partial M; \Zbb) \arrow[r] &
			\cdots
		\end{tikzcd}
	\end{equation}
	
	Assume first that $k > 0$. Proposition \ref{SigmaL} (ii) implies that $\mathcal{F}_L:\partial M\to\Sigma_L$ is a $k$-sphere fiber bundle. Since $S^k$ is connected, the following part of the homotopy exact sequence of the fibration 
	\[
	\begin{tikzcd}
		\cdots \arrow[r] &
		\pi_1(S^k) \arrow[r] &
		\pi_1(\partial M) \arrow[r, "(\mathcal{F}_L)_*"] &
		\pi_1(\Sigma_L) \arrow[r] &
		\pi_0(S^k) \arrow[r] &
		\cdots
	\end{tikzcd}
	\]
	shows that $(\mathcal{F}_L)_\ast : \pi_1(\partial M) \to \pi_1(\Sigma_L)$ is surjective. Hence, $(\mathcal{F}_L)_\ast \allowbreak: H_1(\partial M;\Zbb) \allowbreak \to H_1(\Sigma_L; \Zbb)$ is also surjective. Since $\pi$ is a homotopy equivalence, it follows that the map 
	\begin{align*}
		i_\ast = \pi_\ast^{-1} \circ (\mathcal{F}_L)_\ast : H_1(\partial M;\Zbb) \to H_1(M;\Zbb)
	\end{align*}
	is surjective. Therefore, the exactness of \eqref{longexact} yields $H_1(M,\partial M;\Zbb) = 0$.
	
	Assume now that $k = 0$. By Proposition \ref{SigmaL} (i), $\mathcal{F}_L : \partial M \to \Sigma_L$ is a nontrivial $2$-fold covering map. Since $\pi$ is a homotopy equivalence, exactness of \eqref{longexact} implies that 
	\begin{align*}
		H_1(M, \partial M;\Zbb) \cong \frac{H_1(M;\Zbb)}{\im i_\ast} \cong \frac{H_1(\Sigma_L; \Zbb)}{\im \left( (\mathcal{F}_L)_\ast : H_1(\partial M;\Zbb) \to H_1(\Sigma_L;\Zbb) \right)}.
	\end{align*}
	Set $G = \pi_1(\Sigma_L)$ and $H = (\mathcal F_L)_*\pi_1(\partial M)$. Since $H$ has index two in $G$, $H$ is normal and $G/H \cong \Zbb_2$. In particular, $G/H$ is abelian, so $[G,G] \subset H$. Therefore, the image of $(\mathcal{F}_L)_\ast : H_1(\partial M;\Zbb) \to H_1(\Sigma_L;\Zbb)$ 
	is
	\begin{align*}
		\frac{H}{[G,G]} \subset \frac{G}{[G,G]} = H_1(\Sigma_L;\Zbb).
	\end{align*}
	Consequently,
	\begin{align*}
		\frac{H_1(\Sigma_L;\Zbb)}
		{\im \left((\mathcal F_L)_*:H_1(\partial M;\Zbb)\to H_1(\Sigma_L;\Zbb)\right)} \cong \frac{G/[G,G]}{H/[G,G]} \cong \frac{G}{H} \cong \Zbb_2.
	\end{align*}
	This concludes the proof.
\end{proof}

As an application, one can topologically classify all Zoll surfaces with connected boundary.

\begin{theorem}\label{twodimension}
	Let $(M,g)$ be a Zoll surface with connected boundary whose free boundary geodesics all have length $2L$. Then, one of the following holds:
	\begin{itemize}
		\item[$\mathrm{(i)}$] either $M$ is a disk and all of its free boundary geodesics intersect at a common point at distance $L$ from $\partial M$;
		\item[$\mathrm{(ii)}$] or $M$ is a M\"obius band and none of its free boundary geodesics intersect.
	\end{itemize}
\end{theorem}

\begin{proof}
	If $M$ is a Zoll surface with connected boundary, Theorem \ref{homology2} implies $H_1(M, \partial M;\Zbb)$ is either trivial or $\Zbb_2$. But $H_1(M, \partial M;\Zbb)$ is isomorphic to $H_1(M/\partial M;\Zbb)$, where $M/\partial M$ denotes the space obtained from $M$ by collapsing $\partial M$ to a point. Since the boundary of $M$ is a circle, $M/\partial M$ is homeomorphic to a closed surface obtained by capping the circle with a disk. If the first homology group of a closed surface is trivial, then the surface is topologically a sphere; if it is $\Zbb_2$, the surface is topologically a projective plane. So, $M$ is homeomorphic to either a sphere or a projective plane with a disk removed.
	
	Suppose $M$ is a disk, and let $p,p' \in \partial M$ with $p' = \gamma_p(2L)$. If no other free boundary geodesic of $M$ intersects the image of $\gamma_p$, then all free boundary geodesics starting at a point of one of the intervals determined by $p$ and $p'$ must finish at the same interval. In particular, by continuity of $\mathcal{F}$, at least one free boundary geodesic should start and finish at the same point, which is a contradiction. Therefore, one can find $q \in \partial M$ different from $p$ and $p'$ determining another free boundary geodesic of $M$ whose intersection with the image of $\gamma_p$ is nonempty. By Lemma \ref{distance}, we may assume, without loss of generality, that $\gamma_p(t_0) = \gamma_q(t_0) = x_0$, where $t_0 \leq L$ denotes the first time at which these geodesics intersect. If $\Sigma_{t_0}$ were a curve, it would have two different normal vectors at $x_0$, which cannot happen. Thus, $\Sigma_{t_0}$ is a single point, which implies $t_0 = 2L-t_0$, and then $t_0=L$. Since all free boundary geodesics of $M$ meet $\Sigma_L$ by continuity of $\mathcal{F}$, then all of them intersect each other exactly at $\Sigma_L$.
	
	Now, if $M$ is a M\"obius band then the index of $M$ must be zero by Theorem \ref{homology2} and thus no pair of free boundary geodesics intersect by Corollary \ref{intersect}.
\end{proof}

For the 3-dimensional case, we have the following classification result.

\begin{proposition}[Topological classification of Zoll manifolds with boundary in dimension $3$]\label{thm:class3dim}
	Let $M$ be a $3$-dimensional compact manifold with boundary that admits a Zoll metric of index $k\in\{0,1,2\}$. Then:
	\begin{itemize}
		\item[$\mathrm{(i)}$] If $\partial M$ has 2 connected components (and therefore $k=0$), then $M$ is the product of a closed interval with any closed surface.
		
		\item[$\mathrm{(ii)}$] if $\partial M$ is connected and $k=0$, then the soul can be any non-simply connected closed surface $S$, and $M$ is diffeomorphic to the total space of the twisted interval bundle over $S$. When $S$  is a nonorientable (resp., orientable) closed surface of genus $g$ ($\ge 1$), then the boundary of $M$ is diffeomorphic to the oriented closed surface of genus $g-1$
		(resp., of genus $2g-1$).
		
		\item[$\mathrm{(iii)}$] if $k=1$, then $M$ is diffeomorphic to a solid $3$-torus or to a solid Klein bottle, and the soul is a circle.
		
		\item[$\mathrm{(iv)}$] if $k=2$, then $M$ is a $3$-ball, and the soul consists of a single point.
	\end{itemize}
	All cases above occur.
\end{proposition}
\begin{proof}
	The case when $\partial M$ has two connected components follows immediately from Theorem~\ref{Zolltwoboundaries}.
	
	For the case $k=0$ and $\partial M$ connected, the soul must be a connected and non-simply connected closed surface which admits a one-sided embedding into a $3$-manifold, for otherwise a tubular neighborhood of such embedding would have disconnected boundary. Every closed surface $S$ other than the $2$-sphere admits such embeddings, and a small tubular neighborhood of one of them is the total space of a nontrivial interval bundle over $S$. There exists a unique nontrivial interval bundle over any connected closed surface of genus $g$, see for instance \cite{Hempel1976}. Its boundary is a connected closed orientable surface, and its genus is easily computed as follows. Assume that $S$ is nonorientable of genus $g$, so that its Euler characteristic is $\chi(S)=2-g$. Then, $\chi(\partial M)=4-2g$. Since $\partial M$ is orientable, then its genus is $1-\frac12\chi(\partial M)=g-1$. If $S$ is orientable, then $\chi(S)=2-2g$,
	$\chi(\partial M)=4-4g$, and the genus of $\partial M$ is $1-\frac12\chi(\partial M)=2g-1$.
	
	For the case $k=1$, clearly the soul has to be a circle. It follows from Corollary~\ref{topology} that $\chi(\partial M)=0$, i.e., $\partial M$ is either a torus or a Klein bottle. The only compact $3$-manifolds with this type of boundary are: solid tori, solid Klein bottles, twisted interval bundles over the Klein bottle, or more general nonorientable Seifert fiber spaces with torus boundary. In addition, by Corollary~\ref{thm:fibresballs}, such a manifold must admit a fiber bundle structure over the circle, whose fibers are diffeomorphic to $2$-disks. This is the case only for solid tori and solid Klein bottles, see for instance \cite{Hatcher3Mfd} for a classification of disk bundles over the circle, or \cite{Hempel1976}. Examples of Zoll manifolds diffeomorphic to solid $3$-tori or solid Klein bottles having souls diffeomorphic to the circle are easily constructed.
	
	For the case $k=2$, see Remark~\ref{rem:k=n-1}.
\end{proof}

Along the same lines as Proposition~\ref{thm:class3dim}, item (ii), we conclude with an easy criterion for a closed manifold to be the soul of a Zoll manifold with connected boundary in codimension $1$:

\begin{proposition} \label{thm:soulcodim1connboundary}
	A closed $(n - 1)$-manifold $\Sigma$ is the soul of an $n$-dimensional Zoll manifold with connected boundary and index $k = 0$ if and only if $H^1(\Sigma; \Zbb_2) \neq 0$.
\end{proposition}

\begin{proof}
	Suppose first that $\Sigma$ is the soul of an $n$-dimensional Zoll manifold $M$ with connected boundary and index $k = 0$. By Proposition \ref{SigmaL} (i), the map $\mathcal{F}_L : \partial M \to \Sigma$ is a nontrivial $2$-fold covering map and $(\mathcal{F}_L)_\ast \pi_1(\partial M)$ is a subgroup of index two in $\pi_1(\Sigma)$. So, there exists a nontrivial homomorphism $\chi : \pi_1(\Sigma) \to \Zbb_2$. Since
	\begin{align*}
		H^1(\Sigma; \Zbb_2) \cong \Hom(H_1(\Sigma; \Zbb), \Zbb_2) \cong \Hom(\pi_1(\Sigma),\Zbb_2),
	\end{align*}
	it follows that $H^1(\Sigma; \Zbb_2)\neq 0$.
	
	Conversely, assume that $H^1(\Sigma; \Zbb_2)\neq 0$. Since isomorphism classes of real line bundles over $\Sigma$ are classified by their first Stiefel--Whitney class, there exists a nontrivial real line bundle $\pi : E \to \Sigma$. Choose a bundle metric on $E$ and let $\eta : S(E) \to \Sigma$ be its unit sphere bundle. Since
	$E$ has rank one, $\eta$ is a $2$-fold covering map. Moreover, since $E$ is nontrivial, this covering is nontrivial and therefore connected. Let $ \tau: S(E) \to S(E)$ denote its nontrivial deck transformation. Fix $L > 0$ and define 
	\begin{align*}
		M = \left(S(E) \times [-L, L] \right)/\sim
	\end{align*}
	where
	\begin{align*}
		(x,t)\sim(\tau(x),-t).
	\end{align*}
	This quotient is naturally diffeomorphic to the closed radius $L$ disk bundle of $E$ via the map $[x, t] \mapsto t x$. In particular, $M$ is a smooth compact $n$-manifold with boundary, and $\partial M$ is
	naturally identified with $S(E)$, through the map $S(E) \ni x \mapsto [x,L] \in \partial M$. Since $S(E)$ is connected, so is $\partial M$. Moreover, the image of $(S(E)\times\{0\})$ in $M$ is naturally identified with $S(E)/\langle \tau \rangle = \Sigma$.
	
	Now let $g_\Sigma$ be a Riemannian metric on $\Sigma$. On $S(E)\times[-L,L]$, consider the product metric
	$\widetilde g=\eta^*g_\Sigma \oplus \mathrm dt^2$. Since $\eta \circ \tau = \eta$, the involution
	\begin{align*}
		(x,t)\longmapsto(\tau(x),-t)
	\end{align*}
	is an isometry of $ \left( S(E) \times[-L, L], \widetilde g \right)$. Hence $\widetilde g$ descends to a smooth Riemannian metric $g$ on $M$. For each $x\in S(E)$, the curve $[0, 2L] \ni t \mapsto [x,t - L] \in M$ is the unit-speed free boundary geodesic on $(M,g)$ that issues from $[x, -L] \in \partial M$. It follows that $(M,g)$ is a Zoll manifold with connected boundary and its soul is naturally identified with $\Sigma$. Since this soul has codimension one, the index of $(M,g)$ is $k = 0$.
\end{proof}

\section{Constructions of Zoll manifolds with connected boundary}\label{sec:construction}
Let us now describe a way to obtain Zoll manifolds with boundary from a given one, using a construction known in the literature as the \textit{mapping torus}, which we will now recall.

Let $\varphi : M \to M$ be an isometry of a Riemannian manifold $(M,g)$ (not necessarily Zoll). The group of integers $\Zbb$ acts properly and discontinuously by isometries on the product manifold $(M \times \Rbb, g \oplus \mathrm dt^2)$ via 
\[ n \cdot (p, t) = (\varphi^n(p), t + n).\]
The quotient manifold with the quotient metric is called the {\it mapping torus} of $\varphi$, and is denoted by $(M_\varphi, g_\varphi)$. The following properties are easily verified:
\begin{itemize}
	\item[(i)] $M_\varphi$ is compact if and only if $M$ is compact;
	\item[(ii)] $\partial M_\varphi$ is connected if and only if $\partial M$ is connected;
	\item[(iii)] the natural projection $\pi : M \times \Rbb \to M_\varphi$ is a Riemannian covering map;
	\item[(iv)] the projection $\sigma : M_\varphi \to \Sbb^1$ given by $\sigma([p,t]) \equiv t \bmod 1$ is a fiber bundle over $\Sbb^1$ with fiber $M$.
\end{itemize}

We can now state and prove the following result.

\begin{theorem}\label{mappingtorus}
	Let $(M,g)$ be a Zoll manifold with connected boundary. Given an isometry $\varphi : M \to M$, the mapping torus $(M_\varphi, g_\varphi)$ is a Zoll manifold with connected boundary whose free boundary geodesics have the same length as those of $M$. Moreover, the indices of $M$ and $M_\varphi$ coincide.
\end{theorem}

\begin{proof}
	Given a free boundary geodesic $\gamma_p : [0,2L] \to M$ of length $2L$ and issuing from $p \in \partial M$, and a number $s \in \Rbb$, then $\tilde{\gamma}_{p,s} : [0,2L] \to M \times \Rbb$ given by $\tilde{\gamma}_{p,s}(t) = (\gamma_p(t), s)$ is a free boundary geodesic in $M \times \Rbb$. So, $\pi \circ \tilde{\gamma}_{p,s}$ is also a free boundary geodesic in $M_\varphi$ having length $2L$, since $\pi : M \times \Rbb \to M_\varphi$ is a Riemannian covering map. This shows that the mapping torus is a Zoll manifold with connected boundary whose free boundary geodesics also have length $2L$.
	
	To prove the final statement, denote by $\Sigma_L$ and $\Sigma_L^\varphi$ the souls, respectively, of $M$ and $M_\varphi$. If $\dim M = n$, then
	\begin{align*}
		\dim \Sigma_L &= n-1-k \\
		\dim \Sigma_L^\varphi &= n-k_\varphi,
	\end{align*}
	where $k$ and $k_\varphi$ are the indices of $M$ and $M_\varphi$, respectively. We claim that $\dim \Sigma_L^\varphi = \dim \Sigma_L + 1$. To verify this, first notice that $\varphi(\Sigma_L) = \Sigma_L$, since $\varphi$ is an isometry. Therefore, it makes sense to consider the mapping torus $(\Sigma_L)_{\varphi_L}$ of $\Sigma_L$, where $\varphi_L = \varphi \vert_{\Sigma_L}$. In this case, since $\pi : M \times \Rbb \to M_\varphi$ is a Riemannian covering map,
	\[\dist((\Sigma_L)_{\varphi_L},\partial M_{\varphi}) = \dist(\Sigma_L, \partial M) = L.\]
	Consequently, the fact that the free boundary geodesics of $M_\varphi$ all have length $2L$ implies $(\Sigma_L)_{\varphi_L} = \Sigma_L^\varphi$, which shows our claim. Hence,
	\[1 = \dim \Sigma_L^\varphi - \dim\Sigma_L = k - k_\varphi + 1,\]
	and then $k = k_\varphi$.
\end{proof}

\begin{corollary}\label{thm:anyindex}
	For each $n \geq 2$ and $k \in \{0, \ldots, n-1\}$, there exists an $n$-dimensional Zoll manifold with connected boundary and index $k$.
\end{corollary}

\begin{proof}
	The statement follows from Theorem \ref{twodimension} in the case $n=2$. Suppose the result is true for fixed $n \geq 3$ and any $k \in \{0, \ldots, n-1\}$, i.e., there exists an $n$-dimensional Zoll manifold $M$ with connected boundary and index $k$. Theorem \ref{mappingtorus} implies that, for any value $k_\varphi \in \{0, \ldots, n-1\}$, one can find a mapping torus $M_\varphi$ of index $k_\varphi = k$. Since $\dim M_\varphi = \dim M + 1$, the original assertion follows by induction up to $k_\varphi=n-1$. To conclude the case $k_\varphi = n$, notice that an embedded geodesic ball of an $(n+1)$-dimensional Riemannian manifold is a Zoll manifold with connected boundary and has index $n$.
\end{proof}

\begin{acknowledgements}
We are grateful to Lucas Ambrozio and Daniel Tausk for insightful comments and suggestions. E. Longa was partially supported by S\~ao Paulo Research Foundation grant 2024/01663-8. P. Piccione was partially supported by S\~ao Paulo Research Foundation grant 2022/16097-2 ``Modern Methods in Differential Geometry and Geometric Analysis". R. Santos was partially supported by S\~ao Paulo Research Foundation grant 2023/14796-3 and by National Council for Scientific and Technological Development grant 409513/2023-7.
\end{acknowledgements}

\section*{Declarations}

No datasets were generated or analysed during the current study.

The authors declare no competing interests.

\bibliography{references}

\begin{thebibliography}{10}

\bibitem{Besse}
A.~L. Besse.
\newblock {\em Manifolds all of whose geodesics are closed}, volume~93 of {\em
  Ergebnisse der Mathematik und ihrer Grenzgebiete}.
\newblock Springer-Verlag, Berlin-New York, 1978.

\bibitem{Bolton}
J.~Bolton.
\newblock Transnormal hypersurfaces.
\newblock {\em Proc. Cambridge Philos. Soc.}, 74:43--48, 1973.

\bibitem{Bott}
R.~Bott.
\newblock On manifolds all of whose geodesics are closed.
\newblock {\em Ann. of Math. (2)}, 60:375--382, 1954.

\bibitem{funk}
P.~Funk.
\newblock \"{U}ber {F}l\"{a}chen mit lauter geschlossenen geod\"{a}tischen
  {L}inien.
\newblock {\em Math. Ann.}, 74(2):278--300, 1913.

\bibitem{Gia01}
F.~Giannoni, A.~Masiello, P.~Piccione, and D.~Tausk.
\newblock A generalized index theorem for {M}orse-{S}turm systems and
  applications to semi-{R}iemannian geometry.
\newblock {\em Asian J. Math.}, 5(3):441--472, 2001.

\bibitem{gromollgrove}
D.~Gromoll and K.~Grove.
\newblock On metrics on {$S^{2}$} all of whose geodesics are closed.
\newblock {\em Invent. Math.}, 65:175--178, 1981.

\bibitem{guillemin}
V.~Guillemin.
\newblock The {R}adon transform on {Z}oll surfaces.
\newblock {\em Adv. Math.}, 22(1):85--119, 1976.

\bibitem{Hatcher3Mfd}
A.~Hatcher.
\newblock Notes on basic $3$-manifold topology.
\newblock \url{https://pi.math.cornell.edu/~hatcher/3M/3Mdownloads.html}, 2007.
\newblock Available at the author's webpage.

\bibitem{Hempel1976}
J.~Hempel.
\newblock {\em 3-Manifolds}, volume~86 of {\em Annals of Mathematics Studies}.
\newblock Princeton University Press, Princeton, NJ, 1976.

\bibitem{Kal88}
D.~Kalish.
\newblock The {M}orse index theorem where the ends are submanifolds.
\newblock {\em Trans. Amer. Math. Soc.}, 308(1):341--348, 1988.

\bibitem{Nishikawa1}
S.~Nishikawa.
\newblock Transnormal hypersurfaces: generalized constant width for
  {R}iemannian manifolds.
\newblock {\em Tohoku Math. J. (2)}, 25:451--459, 1973.

\bibitem{Nishikawa2}
S.~Nishikawa.
\newblock Compact two-transnormal hypersurfaces in a space of constant
  curvature.
\newblock {\em J. Math. Soc. Japan}, 26:625--635, 1974.

\bibitem{olsen}
J.~Olsen.
\newblock Three-dimensional manifolds all of whose geodesics are closed.
\newblock {\em Ann. Global Anal. Geom.}, 37(2):173--184, 2010.

\bibitem{PicTau99}
P.~Piccione and D.~Tausk.
\newblock A note on the {M}orse index theorem for geodesics between
  submanifolds in semi-{R}iemannian geometry.
\newblock {\em J. Math. Phys.}, 40(12):6682--6688, 1999.

\bibitem{radeschiwilkingsn}
M.~Radeschi and B.~Wilking.
\newblock On the {B}erger conjecture for manifolds all of whose geodesics are
  closed.
\newblock {\em Invent. Math.}, 210(3):911--962, 2017.

\bibitem{Robertson1}
S.~A. Robertson.
\newblock Generalised constant width for manifolds.
\newblock {\em Michigan Math. J.}, 11:97--105, 1964.

\bibitem{Robertson2}
S.~A. Robertson.
\newblock On transnormal manifolds.
\newblock {\em Topology}, 6:117--123, 1967.

\bibitem{Samelson}
H.~Samelson.
\newblock On manifolds with many closed geodesics.
\newblock {\em Portugal. Math.}, 22:193--196, 1963.

\bibitem{thurston1976}
W.~P. Thurston.
\newblock Existence of codimension-one foliations.
\newblock {\em Ann. of Math.}, 104(2):249--268, 1976.

\bibitem{zoll1903uber}
O.~Zoll.
\newblock {\"U}ber {F}l{\"a}chen mit {S}charen geschlossener geod{\"a}tischer
  {L}inien.
\newblock {\em Math. Ann.}, 57:108--133, 1903.

\end{thebibliography}

\end{document}